\documentclass{article}

\usepackage{amssymb}

\usepackage[utf8]{inputenc}
\usepackage{amsmath}
\usepackage{amsthm}
\usepackage{amsfonts}
\usepackage{amssymb}
\usepackage{tikz}
\usepackage{graphicx}
\usepackage{pgfplots}
\usepackage{xcolor}
\usepackage{enumitem}
\usepackage{mathtools}
\usepackage{cite}
\pgfplotsset{compat=1.11}
\graphicspath{ {./images/} }
\usepackage{float}

\newtheorem{theorem}{Theorem}[section]
\newtheorem{lemma}[theorem]{Lemma}

\theoremstyle{definition}
\newtheorem{definition}[theorem]{Definition}

\newtheorem{algorithm}[theorem]{Algorithm}
\newtheorem{prop}[theorem]{Proposition}

\theoremstyle{remark}
\newtheorem{remark}[theorem]{Remark}

\numberwithin{equation}{section}
\newcommand{\ord}{\operatorname{ord}}

\begin{document}

\title{Extending Lenstra's Primality Test to CM Elliptic Curves and a new Quasi-Quadratic Las Vegas Algorithm for Primality}


\author{Tejas Rao}
\date{December 2022}

\maketitle

\begin{abstract}
    For an elliptic curve with CM by $K$ defined over its Hilbert class field, $E/H$, we extend Lenstra's finite fields test to generators of norms of certain ideals in $\mathcal{O}_H$, yielding a sufficient $\widetilde{O}(\log^3 N)$ primality test and partially answering an open question of Lemmermeyer in the case of CM elliptic curves. Letting $\iota,\gamma, b\in \mathcal{O}_K$, $(\iota)$ prime, and $b$ a primitive $k$-th root of unity modulo $(\iota)^n$ we specialize this test to rational integers of the form $N_{K/\mathbb{Q}}(\gamma\iota^n+b)$ with the norm of $\gamma$ small, giving a Las Vegas test for primality with average runtime $\widetilde{O}(\log^2 N)$, that further certifies primality of such integers in $\widetilde{O}(\log^2 N)$ for nearly all choices of input parameters. The integers tested were not previously amenable to quasi-quadratic heuristic primality certification. 
\end{abstract}

\section{Preliminaries}

Fast primality testing of a rational integer $N$ relies on a theme of Lucas, wherein a finite group is constructed so (provably) large that $N$ must be prime. Pomerance outlines the classical methods arising from this theme \cite{Pom}. The binary modular exponentiation central to this Lucasian theme runs in $\widetilde{O}(\log^2 N)$, providing a hypothetical lower bound to all primality testing barring a new theme. Pomerance proved that this hypothetical minimum bound is met: there exist $\widetilde{O}(\log^2 N)$ certifications of primality for every rational prime \cite{certification,Pom}. \emph{Finding} these certifications, i.e. testing for primality, is another story. The fastest classical test relies on a known factorization of $N^k-1$ for some fixed small $k\in \mathbb{N}$, and runs in heuristic $\widetilde{O}(\log^3 N)$ time:

\begin{theorem}[Lenstra's Finite Fields Test, \cite{Pom,Lens}]

Let $N,k$ positive integers, $N>1$ and $f\in (\mathbb{Z}/N\mathbb{Z})[x]$ monic of degree $k$. Suppose that $F|N^k-1$, $F>\sqrt{N}$, and $F$ has a known prime factorization. If $\exists g\in (\mathbb{Z}/N\mathbb{Z})[x]$ such that in $(\mathbb{Z}/N\mathbb{Z})[x]/(f)$,
\begin{align*}
    &\text{(1) } g^{F}=1,  \\
    &\text{(2) } \gcd(g^{\frac{F}{q}}-1,f)=1, \text{ for each prime } q|F,\\ 
    &\text{(3) } \text{each elementary symmetric polynomial in } g^{N^j}, 0\leq j\leq {k-1}\\& \text{ has coefficients in } \mathbb{Z}/N\mathbb{Z},
\end{align*}
and if none of the residues $N^j\mod F$, $0\leq j\leq k-1$, are proper factors of $N$, then $N$ is prime. 

\end{theorem}

Similar to Pocklington's criterion, multiple bases $g$ may be chosen \cite{Lens}. The algorithm runs in heuristic $\widetilde{O}(k^2\log^3 N)$, and tests based on it are referred to as cyclotomic primality tests \cite{Pombook}. The runtime of this algorithm is deterministically $\widetilde{O}(\log^2 N)$ if $k$ is small and fixed, the number of prime powers of $F$ is polynomial in $\log\log N$, and a suitable $g$ (or multiple bases) is known; it remains heuristic $\widetilde{O}(\log^2 N)$ time even if the bases $g$ are not known beforehand. This is the hypothetical minimum runtime of Lucasian tests, and an algorithm having such (at least heuristic) runtime will henceforth be referred to as an \textit{efficient primality test}. 

Beyond classical primality testing, the theory of Ellipitic Curve Primality Proving (ECPP) has been developed by such figues as Goldwasser, Kilian, Atkin, and Morain \cite{GW,ECPP}. In the seminal Elliptic Curve Primality Proving paper by Atkin and Morain \cite{ECPP}, the theory of complex multiplication (CM) is used to determine the orders of the groups of points of certain elliptic curves and test for primality of any positive integer $N$. An asymptotically-fast version due to Shallit runs in heuristic $\widetilde{{O}}(\log^4 N)$ time and stands as our fastest general algorithm in practice \cite{ECPPFast,EF}. More recently, Milhailescu proposed a variant general primality test running one round of cyclotomic primality testing, followed by a round of ECPP, running in heuristic $\widetilde{O}(\log^3 N)$ time, which would stand as the fastest general algorithm \cite{M}. Importantly, it does not reduce to an \emph{efficient primality test} when the number of prime power factors of $F$ is polynomial in $\log \log N$. 
General testing is not yet $\widetilde{O}(\log^2 N)$, and thus there has been work done to determine heuristic $\widetilde{O}(\log^2 N)$ testing utilizing elliptic curves, including the works of Gurevich and Kunyavskiĭ, Tsumura, Gross, Denomme and Savin, and Chudnovsky and Chudnovsky \cite{1,2,3,4,5}. In addition, Abatzoglou, Sutherland, Wong, and Silverberg use CM elliptic curves to provide a framework for deterministic efficient primality testing for certain sequences of integers not amenable to classical testing, utilizing elliptic curves with CM by the rings of integers of $\mathbb{Q}[\sqrt{-7}],\mathbb{Q}[\sqrt{-15}]$ \cite{suther,sutherpre}. A recent preprint in the same vein proposes an extension of the results to class number $3$ \cite{on}. These works expand the class of rational integers amenable to efficient primality testing, including those of the form $N_{K/\mathbb{Q}}(\gamma\iota^k+1)$ for $\iota,\gamma\in \mathcal{O}_K$, $K$ an imaginary quadratic field, when there is an elliptic curve over the Hilbert class field with CM by $\mathcal{O}_K$ and the norm of $\gamma$ is small. Along with rational integers $N$ such that $N^k-1$ is highly factored into a small number of distinct prime factors using Lenstra's finite fields test, these are the integers most amenable to efficient primality testing in the literature. Some remarks on the state of efficient elliptic curve primality testing are made by Silverberg in \cite{berg}.



In the open problems section of his celebrated work Reciprocity Laws: From Euler to Eisenstein \cite[Appendix C]{lemmer}, Lemmermeyer asks: 
\begin{center}
    Can Lenstra's Primality Test be generalized so as to include 
primality tests based on elliptic curves?
\end{center}
In this paper, we answer this question in the affirmative for certain inputs in the case of CM elliptic curves, producing an analogue heuristic $\widetilde{O}(k^2 \log^3 N)$ primality test for a new class of rational integers that reduces to heuristic $\widetilde{O}(\log^2 N)$ time when $k$ and the number of certain prime power factors is small. To avoid precomputing the complex isogenies of the CM elliptic curve, we introduce another efficient primality test for a smaller new class of rational integers, this time which provides a certificate of primality of a prime $p$ in one trial with probability $1-1/p^{\alpha}$ for some reasonably-sized $\alpha$, and is a Las Vegas primality algorithm with average runtime $\widetilde{O}(\log^2 N)$. The methodology to ensure certification of primality in one trial with high probability is inspired by that of Grau, Marcén, and Sadornil \cite{Grau1,Grau2}. We describe the implementation of the algorithm and give an example of rational integers amenable to primality testing by it. In particular, this Las Vegas primality algorithm can test rational integers in sequences of the form

$$N_{K/\mathbb{Q}}(\gamma\iota^k+b)$$

where $\gamma,\iota\in \mathcal{O}_K$ for some quadratic imaginary field $K$ and $N_{K/\mathbb{Q}}(\iota^n)>N^{1/2+\alpha}$, with the following precomputed information: an elliptic curve $E/H$ with CM by $\mathcal{O}_K$, a rational prime $q=\iota\overline{\iota}$ splitting into two principal ideals over $\mathcal{O}_K$, a primitive $k$-th root of unity $b$ modulo $(\iota)^n$, some $k$-th power non-residue $a$ modulo $\alpha_n=\gamma\iota^n+b$, and that $E/H$ has good reduction modulo an ideal above $(\alpha_n)$ in $H/K$ (it is sufficient to check $\gcd(N_{K/\mathbb{Q}}(\alpha_n),N_{H/\mathbb{Q}}(\operatorname{disc} (E)))=(1)$).

\section{Extending Lenstra's Criterion}

\subsection{Notation and Assumptions} 

Throughout this paper, $p,q$ refer to rational primes. Further, $E/M$ refers to an elliptic curve defined over some number field $M$ and $\mathcal{O}_M$ denotes the ring of integers of $M$. Let $h_M$ denote the class number. For a Dedekind domain (e.g. the ring of integers of a number field) we adopt the definition $\gcd(\mathfrak{j},\mathfrak{i})=\mathfrak{j}+\mathfrak{i}$ for two ideals $\mathfrak{j},\mathfrak{i}$. We note that there is no proper ideal containing both $\mathfrak{j},\mathfrak{i}$ precisely when $\gcd(\mathfrak{j},\mathfrak{i})=(1)$. Further, let $N_{L/V}$ denote the field norm for a field extension $L/V$. 

The number fields of interest for most of this paper are imaginary quadratic fields $K$ with Hilbert class field $H$. We let $E/H$ be an elliptic curve defined over $H\supset K$ with complex multiplication by the ring of integers $\mathcal{O}_K$ of $K$. \emph{Further assume that $E$ has good reduction modulo every relevant ideal unless otherwise stated}. That is, all prime ideals modulo which we reduce $E$ are prime to $\operatorname{disc}(E)$. We further let $\phi_{E,q}$ be the Frobenius endomorphism on the group $E(\overline{\mathbb{F}}_q)$ given by $\phi([x:y:z])=[x^q:y^q:z^q]$. Further we will make the simplifying assumption that $K\neq \mathbb{Q}[\sqrt{-1}],\mathbb{Q}[\sqrt{-3}]$ so that the unit group of $\mathcal{O}_K$ is $\{\pm 1\}$, although it should be noted that much of the theory can be extended without much trouble to these cases.

Let $\mathfrak{N}\subset \mathcal{O}_H$ be an ideal,  $\pi\in \mathcal{O}_K$, and $N\in \mathbb{N}$ such that 

\begin{align*}
    N_{H/K}({\mathfrak{N}})&=\pi\mathcal{O}_K & N_{K/\mathbb{Q}}(\pi\mathcal{O}_K)&=(N)
\end{align*}

We will reduce to the cases where the prime ideal factorization of a large principal ideal factor of $(\pi^k-1)$ is known:

\begin{align*}
    (\pi^k-1) &= \Gamma\Lambda & \Lambda=(\lambda)&=\prod_{\mathfrak{q}|\Lambda} \mathfrak{q}^{e_q}\subset \mathcal{O}_K
\end{align*}

We also consider prime ideals of $\mathcal{O}_H$, $\mathfrak{p}$, with $N_{H/K}({\mathfrak{N}})=\pi_p\mathcal{O}_K$ and $N_{H/\mathbb{Q}}(\mathfrak{p})=p^j$ for a rational positive prime $p$. Also let 

$$f_k=x^k-a\in (\mathcal{O}_H/\mathfrak{N})[x]$$

with $a\in \mathcal{O}_H$ a primitive $k$-th power non-residue modulo $\mathfrak{N}$. In particular, if $\mathfrak{N}$ is prime, then $(\mathcal{O}_H/\mathfrak{N})[x]/(f_k)$ is a degree $k$ finite extension field of $\mathcal{O}_H/\mathfrak{N}$. In addition, reducing modulo $\mathfrak{p}$ (for some prime $\mathfrak{p}|\mathfrak{N}$), we can consider computations in $(\mathcal{O}_H/\mathfrak{p})[x]/(f_k)$ as (potentially not fully reduced) computations in some degree $h<k$ extension field of $\mathcal{O}_H/\mathfrak{p}$. This is because $(f_k)$ splits into the product of irreducible polynomials $g_i$ modulo $\mathfrak{p}$, and thus $(f_k)\subset (g_i)$ for some $g_i$.

Let $P=[x_0:y_0:z_0]\in E(\mathcal{O}_H)$. We write $P\mod \mathfrak{N}$ to denote the coordinate-wise reduction of $P$ modulo $\mathfrak{N}$. Note that we have that $P\equiv O_E\mod \mathfrak{N}\Leftrightarrow z_0\in \mathfrak{N}$. Following the convention of \cite{suther,sutherpre}, we say $P$ is \emph{strongly non-zero} modulo $\mathfrak{N}$ if $\gcd(z_0\mathcal{O}_H,\mathfrak{N})=(1)$. In particular this implies $z_0\mathcal{O}_H$ and $\mathfrak{N}$ are relatively prime. In particular this implies that for each prime $\mathfrak{p}|\mathfrak{N}$, $P\not\equiv O_E\mod \mathfrak{p}$.  

Now if $P=[x_0:y_0:z_0]=[x_0:y_0:\sum_{x=0}^{k-1} a_i x^i]\in E(\mathcal{O}_H[x]/(f_k))$. Note we are implicitly reducing $z_0$ modulo $f_k$ so that $z_0$ is given by an at most degree $k-1$ polynomial; we will continue this practice throughout the paper. In particular this implies that $P\equiv O_E\mod (\mathfrak{N},f_k)\Leftrightarrow P\equiv O_E\mod \mathfrak{N}$ since it is already reduced modulo $f_k$. We thus note 
$$P\equiv O_E\mod (\mathfrak{N},f_k)\Leftrightarrow 
 P\equiv O_E\mod \mathfrak{N} \Leftrightarrow z_0\in \mathfrak{N}\Leftrightarrow a_i\in \mathfrak{N}, 0\leq i\leq k-1$$

Thus we say $P$ is \emph{strongly non-zero} modulo $(\mathfrak{N},f_k)$ if $\gcd(a_i\mathcal{O}_H,\mathfrak{N})=(1)$ for $0\leq i\leq k-1$. In particular this implies that $a_i\not\in \mathfrak{p}$ for each prime $\mathfrak{p}|\mathfrak{N}$ and thus that $P\not\equiv O_E\mod \mathfrak{p}$.

\begin{remark}\label{stronglynonzero}
If one wants to confirm $\gcd(a\mathcal{O}_H,\mathfrak{N})=(1)$, simply show that $\gcd(N_{H/\mathbb{Q}}(a\mathcal{O}_H),N)=1$. This is sufficient, and if $N>\gcd(N_{H/\mathbb{Q}}(a\mathcal{O}_H),N)>1$, $N$ is composite and we may terminate whatever primality testing algorithm we are running. If $N=\gcd(N_{H/\mathbb{Q}}(a\mathcal{O}_H),N)$ and $\gcd(a\mathcal{O}_H,\mathfrak{N})=(1)$ then $\sigma{\mathfrak{N}}|a\mathcal{O}_H$ for some $\sigma'\in Gal(H/K)$, $\sigma'\neq id$. The latter condition may be checked because $\sigma'{\mathfrak{N}}|a\mathcal{O}_H\Leftrightarrow (N)|\prod_{\sigma\in Gal(H/K),\sigma\neq \sigma'} \sigma(a\mathcal{O}_H)\mathfrak{N}$. We can try each guess for $\sigma'\neq id$ to check this. 
\end{remark}

Note that if $E/H$, with CM by $\mathcal{O}_K$, has good reduction at a prime $\mathfrak{p}$, then by Silverman's \emph{Advanced Topics in the Arithmetic of Elliptic Curves}, 
$$\operatorname{End}(E)\rightarrow \operatorname{End}(\widetilde{E})$$
is a homomorphism of $\mathcal{O}_K$-modules, where $\widetilde{E}$ denotes the reduction \cite[Chapter II,IV]{silver2}. Thus it does not make a difference in computations when we mod by $\mathfrak{p}$. If we do computations on $E/H$ modulo a possibly composite ideal $\mathfrak{N}$ with good reduction, then since $\mathfrak{N}\subset \mathfrak{p}$, we can consider computations as partial reductions and thus mod by $\mathfrak{N}$ universally in computation: 

$$[\alpha](Q\mod \mathfrak{N}) \equiv [\alpha]Q \mod \mathfrak{N}$$

\subsection{Hecke Character Properties}

Consider again $E/H$ with CM by $\mathcal{O}_K$ ($K\subset H$). Again recall that throughout this paper $K\neq \mathbb{Q}[i],\mathbb{Q}[\sqrt{-3}]$. We adopt and specialize the following definition-lemma from \cite[Prop. 4.1]{count}.
\begin{lemma}\label{hecke}
    Let $\psi: I(B)\rightarrow K^\times$ denote the \emph{Hecke Character} given as the unique character from the group of fractional ideals with support outside of primes $\beta\in \mathcal{O}_H$ where $E$ has bad reduction satisfying: 
    \begin{align*}
        &\text{(1) $\psi(\mathfrak{p})\in \mathcal{O}_K$, and $\psi(\mathfrak{p})$ is a generator of $N_{H/K}(\mathfrak{p})$}\\
        &\text{(2) $|E(\mathcal{O}_H/\mathfrak{p})|=N_{H/\mathbb{Q}}(\mathfrak{p})+1-Tr_{K/\mathbb{Q}}(\psi(\mathfrak{p}))$}
    \end{align*}
\end{lemma}
\begin{proof}
    We must check that Proposition $4.1$ specializes to this case when we take the order in $K$, $\mathcal{O}=\operatorname{End} E$, to be $\mathcal{O}_K$, and when we let $K\neq \mathbb{Q}[i],\mathbb{Q}[\sqrt{-3}]$. But this is precisely the content of \cite[Lemma 2.6, Remark 2.7, Corollary 4.2]{count}. Notice that conditions $(ii),(iii)$ of \cite[Prop. 4.1]{count} are trivially satisfied in this special case because $\mathcal{O}=\mathcal{O}_K$.
\end{proof}

In particular, $\psi(\mathfrak{p})\in \mathcal{O}_K$ is the Frobenius endomorphism of $E$ defined over ${\mathcal{O}_H/\mathfrak{p}}$, $\psi(\mathfrak{p})[P]=P$ for all $P\in E(\mathcal{O}_H/\mathfrak{p})$ \cite{deuring}. The above paper gives a method of calculating the Hecke character for certain primes $\mathfrak{p}$ in a method that is negligible in computational complexity compared to the runtime of the algorithm \ref{1}, as spelled out below in Definition \ref{comp}, Lemma \ref{easy2}, and \cite[Prop. 6.2]{count}. In particular, we have that 

$$\psi(\mathfrak{p})=u\pi_p$$
for some unit $u\in\mathcal{O}_K^\times$. Since we have removed $\mathbb{Q}[i],\mathbb{Q}[\sqrt{-d}]$, $u=\pm 1$. We define the analogs $\left(\frac{a}{\mathfrak{N}}\right)$ of the Legendre symbol as in \cite[Def. 2.3]{count}, except we allow $\mathfrak{N}$ to be composite in the definition. Let $D$ be the discriminant of $K$. We make the following definition: 
\begin{definition}\label{comp}
     Suppose $E: y^2=x^3+ax+b$. Let $\tau$ be as in \cite[Prop. 5.3]{count}. Recall $\pi:=N_{H/K}(\mathfrak{N})$. Let $\epsilon_\tau$ be as in \cite[Prop 6.2]{count}. For an ideal $\mathfrak{N}\subset \mathcal{O}_K$ prime to $\operatorname{disc}(E)$, define 
    $$\psi(\mathfrak{N})=
     \begin{cases}
        \left(\frac{6b\gamma_3}{\mathfrak{N}}\right)_{2,H}\epsilon_\tau(\pi)\pi & \text{if $D$ is odd} \\
        
        \left(\frac{-6bi\gamma_3}{\mathfrak{N}}\right)_{2,H}\epsilon_\tau(\pi)\pi & \text{if $D\equiv 4,8\mod 16$} \\
        
        \left(\frac{6^2b^2(j-1728)}{\mathfrak{N}}\right)_{4,H}\epsilon_\tau(\pi)\pi & \text{if $D\equiv 0,12\mod 16$}
     \end{cases}
     $$
\end{definition}

Note that by construction we have 

\begin{lemma}\label{easy2}
    If $\mathfrak{N}$ is prime, then $\psi(\mathfrak{N})$ from Definition \ref{comp} and $\psi(\mathfrak{N})$ from Lemma \ref{hecke} agree. 
\end{lemma}
\begin{proof}
    This is Proposition $5.3$ of \cite{count}. 
\end{proof}





\subsection{The $\mathcal{O}_K$-module generated by $P$}

Let $P\in E((\mathcal{O}_H/\mathfrak{N})[x]/(f_k))$, with $f_k\in \mathcal{O}_H[x]$ irreducible modulo each prime factor of $\mathfrak{N}$. We now specify some of the details of the structure of the $\mathcal{O}_K$-module $(P)$ generated by $P$. Let $\operatorname{ord}_{\mathfrak{N}} (P)$ denote the unique ideal such that $[\lambda]P\equiv O_E\mod \mathfrak{N}$ if and only if $\lambda\in \operatorname{ord}_{\mathfrak{N}} (P)$. In other words, $\ord_{\mathfrak{N}} (P)$ is the annihilator of $(P)$. We must check this is well defined, and that we can say something about it computationally. 

\begin{lemma}\label{ord} 
If $\mathfrak{N}$ is prime, $\ord_{\mathfrak{N}} (P)$ exists. Moreover, for some unit $u\in \mathcal{O}_K^\times$,  
\begin{align*}
\operatorname{ord}_{\mathfrak{N}} (P)\supset (u{N}_{H/K}(\mathfrak{N})^k-1)=(\psi(\mathfrak{N})^k-1). 
\end{align*}
\end{lemma} 

To do this we need the help of two lemmata. 

\begin{lemma}\label{size}
    For all $P\in E((\mathcal{O}_H/\mathfrak{N})[x]/(f_k))$, $(\psi(\mathfrak{N})^k-1)P=O_E$ if $\mathfrak{N}$ is prime. 
\end{lemma}
\begin{proof}
    By definition, $\psi(\mathfrak{N}))[x_0:y_0:z_0]\equiv [x_0^n:y_0^n:z_0^n] \mod\mathfrak{N}$, where $n=\#\mathcal{O}_H/\mathfrak{N}$. We can write each projective coordinate as some polynomial $\sum_{i=0}^{k-1} a_i x^i$ in $(\mathcal{O}_H/\mathfrak{N})[x]/(f_k)$, with $a_i\in \mathcal{O}_H/\mathfrak{N}$. By the definition of $f_k$ and a Galois extension, $(\sum_{i=0}^{k-1} a_i x^i)^n=\sum_{i=0}^{k-1} a_i \zeta_{k}^i x^i$ for some $k$-th root of unity $\zeta_k\in \mathcal{O}_H$. Then lemma follows when raising to the $n$-th power $k$ times, since $[\psi(\mathfrak{N})^k]P\equiv P\mod \mathfrak{N}$. 
\end{proof}

Now we introduce \ref{annihilator}, using a similar methodology to \cite{suther} theorem 3.5 (a). 


\begin{lemma}\label{annihilator}
If $P\not\equiv O_E\mod \mathfrak{N}$ and $[\mathfrak{a}]P\equiv O_E\mod \mathfrak{N}$ for some ideal $\mathfrak{a}$ (if $[\lambda]P\equiv O_E\mod \mathfrak{N}$ for each $\lambda\in \mathfrak{a}$), and if there is an element $\lambda\in \frac{\mathfrak{a}}{\mathfrak{h}}$ such that $[\lambda]P\not\equiv O_E\mod \mathfrak{N}$, for each prime $\mathfrak{h}|\mathfrak{a}$, then $$[\lambda]P\equiv O_E\mod \mathfrak{N}\Leftrightarrow \lambda\in \mathfrak{a}.$$ 
\end{lemma} 
\begin{proof}
Assume that $[\lambda]P\equiv O_E\mod \mathfrak{N}$. Then further assuming $\lambda\not \in \mathfrak{a}$, we have that $[\gcd(\mathfrak{a},\lambda End(E))]P\equiv O_E\mod \mathfrak{N}$. If $\gcd(\mathfrak{a},\lambda End(E))=(1)$ then we have a contradiction. So assume that $\gcd(\mathfrak{a},\lambda End(E))$ is a proper ideal of $End(E)$. But then since $\gcd(\mathfrak{a},\lambda End(E))|\mathfrak{a}$ and since $\lambda\not\in \mathfrak{a}$ by assumption, we have that $\gcd(\mathfrak{a},\lambda End(E))\supsetneq \mathfrak{a}$, a contradiction because then for some prime $\mathfrak{h}|\mathfrak{a}$, $\gcd(\mathfrak{a},\lambda End(E))\supset \frac{\mathfrak{a}}{\mathfrak{h}}$. 
\end{proof}

\begin{proof}[Proof of Lemma \ref{ord}]
If $\mathfrak{N}$ is prime then $[\lambda]P\equiv O_E\mod \mathfrak{N}$ for every $\lambda\in (\psi(\mathfrak{N}))^k-1)$, by Lemma \ref{size}. If $[1]P=O_E\mod \mathfrak{N}$, then $(1)=\operatorname{ord}_{\mathfrak{N}} (P)$, and the condition that $[\lambda]P\not\equiv O_E\mod \mathfrak{N}$ for $\lambda\not\in (1)$, as well as uniqueness, is trivial. Otherwise we have that $[1]P\not\equiv O_E\mod \mathfrak{N}$. Then consider each prime $\mathfrak{h}|(\psi(\mathfrak{N})^k-1)$, and the corresponding ideal $\mathfrak{h}^{-1}(\psi(\mathfrak{N})^k-1)=\frac{(\psi(\mathfrak{N})^k-1)}{\mathfrak{h}}$. If for each $\mathfrak{h}$ we have that $[\lambda]P\not\equiv O_E\mod \mathfrak{N}$ for some $\lambda\in \mathfrak{h}^{-1}(\psi(\mathfrak{N})^k-1)$, then we have $\operatorname{ord}_{\mathfrak{N}} (P)=(\psi(\mathfrak{N})^k-1)$ by Lemma \ref{annihilator}. Alternatively, let $i$ index through the distinct $\mathfrak{h}_i|(\psi(\mathfrak{N})^k-1)$ such that $[\mathfrak{h}_i^{-1}\psi(\mathfrak{N})^k-1)]P\equiv O_E\mod \mathfrak{N}$. Then $\operatorname{ord}_{\mathfrak{N}} (P)=\frac{(\psi(\mathfrak{N})^k-1)}{\prod_i \mathfrak{h}_i}$ by construction and Lemma \ref{annihilator}. This divides $(\psi(\mathfrak{N})^k-1)$ and is unique again by Lemma \ref{annihilator}.



\end{proof}




Now let $\mathfrak{N}=\mathfrak{p}^n$ be a prime a power. By construction of $f_k$, $f_k$ is irreducible modulo $\mathfrak{p}$. Letting $R=(\mathcal{O}_H/\mathfrak{N})[x]/(f_k)$, we see that $R/\mathfrak{p}=(\mathcal{O}_H/\mathfrak{p})[x]/(f_k)$ is a field. Thus $\mathfrak{p}$ is a maximal ideal of $R$. From \cite{lensalg} we have 

\begin{lemma}
Let $R$ be a finite ring and $E/R$ an elliptic curve. The obvious projection map of groups $\phi: E(R)\rightarrow E(R/\mathfrak{m})$ is  a surjection with $\# \ker (\phi)=\# \mathfrak{m}$. 
\end{lemma}
It should be emphasized that $\phi$ is a map of \emph{groups}. By the discussion above we can take $R=(\mathcal{O}_H/\mathfrak{N})[x]/(f_k)$. Now let $Q$ satisfy $\phi(Q)=O_E$. Then by the definition of a group hom, $\phi(Q\oplus Q...\oplus Q)=O_E$. Since $\# \mathfrak{p}=p^{n-1}$, and since the size of $(Q)$ (the subgroup of $\ker (\phi)$ generated by $Q$) is the number of unique sums, we have that Lagrange's theorem implies $\# (Q)|p^{n-1}$ and thus that the number of unique sums divides $p^{n-1}$. Note that $\phi(Q)=O_E$ implies by construction of $\phi$ that $Q\equiv 0\mod \mathfrak{p}$, and that the number of unique sums is the first $s$ such that $Q\oplus..._s\oplus Q\equiv O_E\mod \mathfrak{N}$. Further consider that for all $P$, $(\psi(\mathfrak{p})^k-1)P=O_E$ in $R$ by Lemma \ref{ord}. Putting this together: 
\begin{lemma}\label{poword}
Let $\mathfrak{N}=\mathfrak{p}^{n}$ be a prime power. Then $\ord_\mathfrak{N}(P)\supset p^{n-1}(\psi(\mathfrak{N})^k-1)$ exists. 
\end{lemma}
\begin{proof}
    Use the above discussion and the proof of Lemma \ref{ord}, this time with $p^{n-1}(\psi(\mathfrak{N})^k-1)$ in place of $(\psi(\mathfrak{N})^k-1)$. 
\end{proof}

Now let $\mathfrak{N}$ be an arbitrary ideal. Iff $P\equiv O_E\mod \mathfrak{p^n}$ for each prime power dividing $\mathfrak{N}$, then $P\equiv O_E\mod \mathfrak{N}$. Thus 
\begin{lemma}\label{compord}
In general, 
$$\ord_\mathfrak{N}(P)\supset \operatorname{lcm}_{\mathfrak{p}^n|\mathfrak{N}} p^{n-1}(\psi(\mathfrak{N})^k-1)\supset \prod_{\mathfrak{p}^n|\mathfrak{N}} p^{n-1}(\psi(\mathfrak{N})^k-1)$$ exists. 
\end{lemma}
\begin{proof}
    Use the above discussion and the proof of Lemma \ref{poword}, along with the proof of Lemma \ref{ord}, this time with $\operatorname{lcm}_{\mathfrak{p}^n|\mathfrak{N}} p^{n-1}(\psi(\mathfrak{N})^k-1)$ in place of $(\psi(\mathfrak{N})^k-1)$. 
\end{proof}

\subsection{Main Theoretical Results}

In this section assume that $E/H$ has good reduction modulo $\mathfrak{N}$. To begin this section we should remark the following. 
\begin{remark}\label{rem}
    For $\mathfrak{p}\subset \mathcal{O}_H$, we have that $N_{H/\mathbb{Q}}(\mathfrak{p})=p^j$ is a prime power. Thus if we first check that the integers $N$ we test are not perfect powers, we can test the primality of $N=N_{H/\mathbb{Q}}(\mathfrak{N})$ by testing the primality of $\mathfrak{N}$.
\end{remark}

\begin{remark}
    The trial division steps in this section will be proven in the proceeding one. Additionally, an algorithm to complete the trial division will be given. 
\end{remark}
We now state our analogue of Lenstra's theorem for CM elliptic curves. Since we cannot directly check that $f_k$ is irreducible modulo all prime factors $\mathfrak{p}$ of $\mathfrak{N}$, we note that we can relax the assumption that $f_k$ is irreducible modulo each prime factor:

\begin{lemma}\label{reduce}
    Let $f_k=x^k-a$ for an arbitrary element $a\in \mathcal{O}_H$ with $a\not\in \mathfrak{p}$. Let $P=[x_0:y_0:z_0]\in E((\mathcal{O}_H/\mathfrak{p})[x]/(f_k))$. Then $[\psi(\mathfrak{p})] P\equiv [\sigma_g x_0: \sigma_g y_0: \sigma_g z_0]\mod \mathfrak{p}$ for $\sigma_g(\sum a_i x^i)=\sum a_i \zeta^i_g x^i$ for $\zeta_g$ some $k$-th root of unity modulo $(\mathfrak{p},f_k)$. In particular, $\ord_\mathfrak{p}([P])|\psi(\mathfrak{p})^k-1$. 
\end{lemma}
\begin{proof}
    By standard theory, the Hecke character $\psi(\mathfrak{p})$ is a well-defined endomorphism of $E((\mathcal{O}_H/\mathfrak{p})[x]/(g))$ for an irreducible factor $g$ of $f_k$ modulo $\mathfrak{p}$. Consider $P$ as a representative of the equivalence class $[P]=(P\mod (g))$ in $E((\mathcal{O}_H/\mathfrak{p})[x]/(g))$. Then we may compute a representative of $[\psi(\mathfrak{p})] [P]$ as $[\psi(\mathfrak{p})]P$. Let $P=[x_0:y_0:z_0]$, then for $p^j=\# \mathcal{O}_H/\mathfrak{p}$, $[\psi(\mathfrak{p})]P=[x_0^{p^j}:y_0^{p^j}:z_0^{p^j}]=[\sigma_g x_0:\sigma_g y_0: \sigma_g z_0]\mod \mathfrak{p}$, since each coordinate may be expressed as $\sum a_i x^i$ in $(\mathcal{O}_H/\mathfrak{p})[x]/(f_k)$ (use the binomial theorem modulo a prime and choice of $f_k$). Since $[\psi(\mathfrak{p})^k-1]P\equiv O_E\mod \mathfrak{p}$, the latter conclusion follows by the proof of Lemma \ref{ord}. 
\end{proof}

For the following theorem, let $\rho_j$ be a map given by $\rho_j(\sum a_i x^i)=\sum a_i \zeta_j^i x^i$ for a $\zeta_j$ a \emph{primitive} $k$-th root of unity modulo $(\mathfrak{N},f_k)$. Further (after reducing $\zeta_j^m-1$ modulo $f_k$) let $\gcd(\mathfrak{N},\zeta_j^m-1)=(1)$ for $1\leq j<k$. This ensures that $\zeta_j$ is a primitive $k$-th root of unity modulo $(\mathfrak{p},f_k)$. 

\begin{theorem}\label{lens} 
Let notation be as above and fix some $\mathfrak{N}\nmid (2)$. Assume that $N>1\neq n^r$ for some integer $n$ and $r>1$.  Let $(\psi(\mathfrak{N})^k-1)=\Gamma \Lambda$, with $\Lambda=(\lambda)$ principal, the primary factorization $\prod \mathfrak{q}^{e_q}$ of $\Lambda$ known, and ${N_{K/\mathbb{Q}}(\Lambda)}>N^{1/2}$. If one can choose an $f_k=x^k-a$ with $a$ a primitive $k$-th non-residue modulo $\mathfrak{N}$, and if for each $\mathfrak{q}$, $\exists P_q=[x_q:y_q:z_q]\in {E}((\mathcal{O}_F/\mathfrak{N})[x]/(f_k))$ such that 
\begin{align*}
&\text{(1) } [\lambda]P_q=O_E,\\
&\text{(2) } [\lambda_{\mathfrak{q}}]P_q \text{ is strongly nonzero modulo $(\mathfrak{N},f_k)$}, \text{ where $\lambda_\mathfrak{q}$}\\ &\text{is some element of ${\Lambda}/{\mathfrak{q}}$},\\ 
&\text{(3) } [\psi(\mathfrak{N})^m]P_q=[\rho_j^m x_q: \rho_j^m y_q: \rho_j^m z_q], \text{ for some $j$ and $1\leq m\leq k$,}
\end{align*}
then $\psi(\mathfrak{p})=\psi(\mathfrak{N})^m$ in $\mathcal{O}_K/(\Lambda)$ for $m=1,...,k$.  If further none of the $O(4(-d)+4d^2)$ residues $\beta$ of $\psi(\mathfrak{N})^m$, $m=1,...,k$ with $N_{K/\mathbb{Q}}(\beta)\leq N_{K/\mathbb{Q}}(\Lambda)$ have that $N_{K/\mathbb{Q}}(\beta)$ properly divides $N$, then $N$ is prime. 

\end{theorem}




\begin{proof}
Assume that conditions $(1)$ and $(2)$ hold for some prime divisor $\mathfrak{q}|\Lambda$. Reducing modulo each $\mathfrak{p}$, we can consider our calculations done over the field $(\mathcal{O}_H/\mathfrak{N})[x]/(g)$ for some $g|f_k$ irreducible modulo $\mathfrak{p}$. Condition $(1)$ yields that $\operatorname{ord}_{\mathfrak{p}}(P_q)\supset \mathfrak{q}^{v_\mathfrak{q}(\Lambda)}$ by definition of order, for each prime ideal $\mathfrak{p}|\mathfrak{N}$. Condition $(2)$ yields that there is an element $\lambda_\mathfrak{q}$ of $\Lambda/\mathfrak{q}$, and thus of $\mathfrak{q}^{v_\mathfrak{q}(\Lambda)-1}$ such that $[\lambda]P_q\not\equiv O_E\mod \mathfrak{p}$. Thus $\operatorname{ord}_{\mathfrak{p}}(P_q)\not\supset \mathfrak{q}^{v_\mathfrak{q}(\Lambda)-1}$ by definition of order. So we have that $\operatorname{ord}_{\mathfrak{p}}(P_q)\subset \mathfrak{q}^{v_\mathfrak{q}(\Lambda)}$ by primality. Note this holds for each $\mathfrak{q}|\Lambda$.

By condition $(3)$, $[\psi(\mathfrak{N})^m]P_q=[\rho_{j}^mx_q:\rho_{j}^my_q:\rho_{j}^mz_q]$, $1\leq m<k$. Writing $z_q=\sum a_i x^i$, by construction of $\rho_j$, 

$$\rho_{j}^m z_q\equiv \sum a_i \zeta^{im}_j x^i\not\equiv z_q\mod \mathfrak{p}, 1\leq m<k$$
since $\zeta_j$ is a primitive $k$-th root of unity modulo $\mathfrak{p}$. Further $\rho_{j}^k z_q\equiv z_q\mod \mathfrak{p}$ (and similarly for $x_q,y_q$). By Lemma \ref{reduce}, $[\psi(\mathfrak{p})]P_q\equiv [\sigma_g x_q:\sigma_g y_q:\sigma_g z_q]\mod \mathfrak{N}$. In particular $\sigma_g z_g=\sum a_i \zeta_g^i x^i$ for some $k$-th root of unity modulo $(\mathfrak{p},f_k)$ and so  
\begin{align*}
[\psi(\mathfrak{N})^m]P_q&\equiv [\rho_{j}^m x_q:\rho_{j}^m y_q:\rho_{j}^m z_q]\\
&\equiv [\rho_j^m x_q:\rho_j^m y_q:\sum a_i \zeta^{im}_j x^i]\\
&\equiv [\sigma_{g} x_q:\sigma_{g} y_q:\sigma_{g} z_q]\equiv [\psi(\mathfrak{p})]P_q \mod \mathfrak{p}
\end{align*} 
for $1\leq m<k$ with $\zeta_j^{m}\equiv \zeta_g\mod \mathfrak{p}$ (such an $m$ exists by primitivity). In particular we know that $[\psi(\mathfrak{N})^m-\psi(\mathfrak{p})]P_q\equiv O_E\mod \mathfrak{p}$. By the results of the first paragraph and Lemma \ref{ord}, $\psi(\mathfrak{N})^m-\psi(\mathfrak{p})\in \mathfrak{q}^{e_q}$. Since this is true for each prime $\mathfrak{q}|\Lambda$, $\psi(\mathfrak{N})^m-\psi(\mathfrak{p})\equiv 0\mod \Lambda$. 

Assume now that none of the residues $\beta$ of $\psi(\mathfrak{N})^m$, $m=1,...,k$ modulo $\Lambda$ with $N_{K/\mathbb{Q}}(\beta)\leq N^{1/2}$ have $N_{K/\mathbb{Q}}(\beta)|N$. Then for every distinct prime divisor $\mathfrak{p}|\mathfrak{N}$, $N_{K/\mathbb{Q}}(\psi(\mathfrak{N}))> N^{1/2}$ and $N_{K/\mathbb{Q}}(\psi(\mathfrak{N}))|N$. Clearly there can thus be only one distinct prime divisor $\mathfrak{p}$ of $\mathfrak{N}$. If $\mathfrak{p}^r=\mathfrak{N}$ for $r>1$, then by norm multiplicativity, $N=p^r$, which is a contradiction since we assumed $N$ was not a prime power. Thus $\mathfrak{N}$ is prime. Since we assumed $N$ is not a prime power, $N$ is also prime by Remark \ref{rem}. 

By Theorem \ref{resi}, there are $O(4(-d)+4d^2)$ residues to check for each given $m$. 

\end{proof}



In practice, an issue that arises with condition $(3)$ of the initial test is that one must compute the action of the (non-integer) complex multiplication isogenies in precomputation. There is no repository of such isogenies known to the author online. Additionally, the test, although a sufficient condition for primality, is not a necessary one, and one that runs in $\widetilde{O}(\log^3 N)$ for fixed $k$, see the analysis of Lemma \ref{runtime}, Remark \ref{a} (an asymptotically similar runtime to Lenstra's criterion). We will now introduce the framework for a Las Vegas test for primality that runs in average time $\widetilde{O}(\log^2 N)$ on certain classes of integers, serves as an \emph{efficient primality test}, and certifies primality in one run for nearly all choices of input points.

We begin with some lemmata. We say a solution to $x^k=1$ is \emph{$k$-primitive} if $x^m\neq 1$ for $0<m<k$. We have the following well-known lemma. 

\begin{lemma}\label{sim} 
Let $\mathfrak{p}$ be a prime ideal which is not inert and $\mathfrak{p}\nmid 2$. Then $\mathcal{O}_K/\mathfrak{p}^n$ is generated by one element when considered as a multiplicative group. In particular, if $x^k=1$ in $\mathcal{O}_K/\mathfrak{p}^n$ has a $k$-primitive solution, then it has precisely $k$ solutions. 
\end{lemma}





\begin{lemma}\label{pdiv} 
Let $N=\psi\overline{\psi}$ split in $\mathcal{O}_K$. Let $n$ be a positive integer such that $(n)|\psi-1$, then $N\equiv 1\mod n$. 
\end{lemma}
\begin{proof}
If $(n)|\psi-1$, then $(n=\overline{n})|(\overline{\psi-1}=\overline{\psi}-1)$. But then $N=\psi\overline{\psi}\equiv 1\cdot 1=1\mod n$, as desired. 
\end{proof}

Thus we do not lose much by taking $\mathfrak{q}$ to be non-inert. If it were inert in the following theorem, the classical finite fields test could be used. Further, it does not hurt to assume $(q^{e_q})\nmid (\psi(\mathfrak{N})^{2k}-1)$, as otherwise $q^{e_q}|N^{2k}-1$ by Lemma \ref{pdiv} and the classical Lenstra primality test may be used. 

\begin{theorem}\label{fast}
Let notation be as above and fix some $\mathfrak{N}\nmid (2)$. Assume $N>1$. Let $(\psi(\mathfrak{N})^k-1)=\Gamma\mathfrak{q}^{e_q}$ with $\mathfrak{q}\nmid (2)$ a non-inert principal prime and $q^{e_q}:=N_{K/\mathbb{Q}}(\mathfrak{q}^{e_q})>N^{1/2}$. Further assume $((\psi(\mathfrak{N}))^m-1)\not\in \mathfrak{q}^{e_q}$ for $0< m<k$, and that $(q^{e_q})\nmid ((\psi(\mathfrak{N}))^{2k}-1)$. If there is an $f_k=x^k-a$ with $a$ a primitive $k$-th non-residue modulo $\mathfrak{N}$, and if $\exists P\in E((\mathcal{O}_H/\mathfrak{N})[x]/(f_k))$ such that 
\begin{align*}
    \text{(1) } [q^{e_q}]P&=O_E,\\
    \text{(2) } [q^{e_q-1}]P &\text{ is strongly nonzero modulo $\mathfrak{N}$},
\end{align*}
then $\psi(\mathfrak{p})=\psi(\mathfrak{N})^m$ in $\mathcal{O}_K/\mathfrak{q}^{e_q}$ for some prime $\mathfrak{p}|\mathfrak{N}$ up to units. If further none of the $O(4(-d)+4d^2)$ residues $\beta$ of $\psi(\mathfrak{N})^m$, $m=1,...,k$ with $N_{K/\mathbb{Q}}(\beta)\leq N_{K/\mathbb{Q}}(\Lambda)$ have that $N_{K/\mathbb{Q}}(\beta)$ properly divides $N$, then $N$ is prime.  
\end{theorem}


\begin{proof}
Condition $(1)$ gives $\ord_\mathfrak{p}(P)\supset ({q}^{e_q})$ by definition, for every prime $\mathfrak{p}|\mathfrak{N}$. Condition $(2)$ yields that there is an element $\lambda$ of $({q}^{e_q-1})$ such that $[\lambda]P\not\equiv O_E\mod \mathfrak{p}$ for each prime $\mathfrak{p}|\mathfrak{N}$. This implies that $\ord_\mathfrak{p}(P)\not\supset ({q}^{e_q-1})$ and thus that $\mathfrak{q}^{e_q}|\ord_\mathfrak{p}(P)$ or $\overline{\mathfrak{q}}^{e_q}|\ord_\mathfrak{p}(P)$ for each $\mathfrak{p}$. Assume for contradiction that $\overline{\mathfrak{q}}^{e_q}| \ord_\mathfrak{p}(P)|(\psi(\mathfrak{p})^k-1)$ for each $\mathfrak{p}$ (Lemma \ref{ord}). By Theorem 5.3 of \cite{count}, $\pi_\mathfrak{p}^k=u_p\psi(\mathfrak{p})^k$ for some unit $u_p\in \mathcal{O}_K^\times$. This implies that $\overline{\mathfrak{q}}^{e_q}|(u_p\pi_\mathfrak{p}^k-1)$ and thus that 
$\overline{\mathfrak{q}}^{e_q}|(u_p^{v_p}\pi_\mathfrak{p}^{v_pk}-1)$ where $v_p$ is the highest power of $\mathfrak{p}$ dividing $\mathfrak{N}$ by norm multiplicativity. Again by multiplicativity, we see that 
\begin{align*}
    \overline{\mathfrak{q}}^{e_q}|(u_p^{v_p}\pi_\mathfrak{p}^{v_pk}-1)&\Rightarrow \overline{\mathfrak{q}}^{e_q}|(u\pi^k-1)
\end{align*}
for some unit $u$. By construction, $\psi(\mathfrak{N})=u'\pi$ for some unit $u'$, and thus 
\begin{align*}
    \psi(\mathfrak{p})^k-1&=u'^k\pi^k-1
\end{align*}
Note $u'=\pm 1$. If $u'=1$, then since $\overline{\mathfrak{q}}^{e_q}|\psi(\mathfrak{N})^k-1$, we have $q^{e_q}|\psi(\mathfrak{N})^k-1$, a contradiction by assumption. But if $u''=-1$ then $\overline{\mathfrak{q}}^{e_q}|\psi(\mathfrak{N})^k+1$. This implies that $q^{e_q}|(\psi(\mathfrak{N})^k+1)(\psi(\mathfrak{N})^k-1)=\psi(\mathfrak{N})^{2k}-1$ by construction, a contradiction to the assumptions in the theorem. Thus there is some $\mathfrak{p}|\mathfrak{N}$ such that $\mathfrak{q}^{e_q}|\ord_{\mathfrak{p}}(P)$. Fix this $\mathfrak{p}$.

We have by definition of $f_k$ and Lemma \ref{reduce} that $[\psi(\mathfrak{p})^k]P=P \mod \mathfrak{p}$ and thus that $\psi(\mathfrak{p})^k=1$ in $\mathcal{O}_K/\mathfrak{q}^{e_q}$ by definition of order as the annihilator (Lemma \ref{ord}). Since we have by choice that $\psi(\mathfrak{N})^k=1$ and $\psi(\mathfrak{N})^m\neq 1$ in $\mathcal{O}_K/\mathfrak{q}^{e_q}$ for $0<m<k$, and since by Lemma \ref{sim} there are precisely $k$ elements $e$ in $\mathcal{O}_K/\mathfrak{q}^{e_q}$ with $e^k=1$, we have that said $k$ elements are precisely $\psi(\mathfrak{N})^m$, $0\leq m<k$. In particular, $\psi(\mathfrak{p})=\psi(\mathfrak{N})^m$ in $\mathcal{O}_K/\mathfrak{q}^{e_q}$ for some $0\leq m<k$.

The result then follows exactly as in Theorem \ref{lens}. 

\end{proof}







To specify a quasi-quadratic Las Vegas algorithm from this theorem, we do the following discussion. Assume $\mathfrak{q}^x|\psi(\mathfrak{p})^k-1$. By \cite{wash}, we have that, as groups,

$$E((O_H/\mathfrak{p})[x]/(f_k))\simeq_\phi \mathbb{Z}/n\mathbb{Z}\times \mathbb{Z}/m\mathbb{Z}$$ 

where $n|m$ and the groups are additive. Denote $\#E_{p,k}=\#E((O_H/\mathfrak{p}))[x]/(f_k)$. Assume $q^x|\#E_{p,k}$. By primality we have that $q^x|nm$ implies $q^{x-y}|n$, $q^y|m$. Since $n|m$, we must have that $y\geq x-y$, so we can take $y\geq x/2$. Thus we have 

$$\left[\dfrac{\#E_{p,k}}{q^x}\right]P=O_E\Leftrightarrow (P\mapsto_\phi (a,b): \text{ $n|a\cdot \frac{\#E_{p,k}}{q^x}$ and $m|b\cdot \frac{\#E_{p,k}}{q^x}$})$$

Note that since $n,m|\#E_{p,k}$ ($nm=\#E_{p,k}$),

$$(P\mapsto_\phi (a,b): \text{ $n|a\cdot \frac{\#E_{p,k}}{q^x}$ and $m|b\cdot \frac{\#E_{p,k}}{q^x}$})\Leftrightarrow (P\mapsto_\phi (a,b): \text{ $q^{x-y}|a$ and $q^y|b$})$$

Wlog, $0\leq a\leq n, 0\leq b\leq m$. There are $\lfloor{m/q^y}\rfloor+1$ such $b$ that are multiples of $q^y$. This implies that randomly choosing $(a,b)\in \mathbb{Z}/n\mathbb{Z}\times \mathbb{Z}/m\mathbb{Z}$, the probability that $q^y|b$, $q^{x-y}|a$ (which is less than or equal to the probability that $q^y|b$) is at most 
\begin{align*}
    \dfrac{n(\lfloor{m/q^y}\rfloor+1)}{nm}&\leq \dfrac{n{m/q^y}+n}{nm}\\
    &=\dfrac{1}{q^y}+\dfrac{1}{m}\\
    &\leq \dfrac{2}{q^{x/2}}.
\end{align*}

We can now show the following theorem. 

\begin{theorem}\label{prob}
Let $\mathfrak{N}$ be prime and satisfy the assumptions of Theorem \ref{fast} and let $f_k$, $E/H$ also be as in Theorem \ref{fast}, and further suppose $(\psi(\mathfrak{N})^k-1)=\Gamma\mathfrak{q}^{e'}$ with $\mathfrak{q}$ principal and $q^{e'}>N^{1/2+\alpha}$. Then for a randomly chosen $Q$,  conditions $(1),(2)$ are satisfied for 

$$P=\left[\dfrac{\#E_{p,k}}{q^{e'}}\right]Q=\left[\dfrac{(\psi(\mathfrak{N})^k-1)(\overline{\psi(\mathfrak{N})}^k-1)}{q^{e'}}\right]Q$$

 with probability at least $1-\frac{1}{N^{\alpha/2}}$ for some  $e_q$ such that $N^{1/2}<q^{e_q}\leq q^{e'}$.

\end{theorem}
\begin{proof}
By the above discussion, if $\mathfrak{N}$ is prime then 
\begin{align*}
    \left[q^{x}\right]P&\equiv O_E\mod \mathfrak{N}\\
    \Leftrightarrow \left[\dfrac{\#E_{p,k}}{q^{e'-x}}\right]Q&\equiv O_E\mod \mathfrak{N}
\end{align*}
occurs with probability at most $1/q^{e'-x}$. If we let $x=\lfloor{N^{1/2}/q}\rfloor$, then $q^x$ is the largest power of $q$ less than or equal to $N^{1/2}$ and $q^{e'-x}>N^{\alpha}$. Thus $[q^x]P\equiv O_E\mod \mathfrak{N}$ with probability at most $1/N^{\alpha/2}$ by the above discussion. If it is not the identity, then by the Lemma \ref{ord} and that $(q^{e'})\subset \ord_\mathfrak{N}(P)$ and $(q^{x})\not\subset \ord_\mathfrak{N}(P)$, $q^{e_q}=\ord_\mathfrak{N}(P)$ with $q^{e_q}>N^{1/2}$ by definiton of $x$, satsifying condition $(1)$ for $e_q$. Condition $(2)$ is satisfied for $e_q-1$ since by choice of $e_q$, $[q^{e_q-1}]P\not\equiv O_E\mod \mathfrak{N}$, and since $\mathfrak{N}$ is prime, this yields strongly-nonzero modulo $\mathfrak{N}$. 
\end{proof}

If the test fails to certify primality or prove compositness in one run through with a non-negligible $\alpha$, then the number is very likely composite, so utilize the Miller-Rabin compositness test, which has an average runtime of $\widetilde{O}(\log^2 N)$ \cite{mill,mill2,mill3}. As will be shown in the Implementation section, this provides a quasi-quadratic Las Vegas algorithm for primality of a new class of integers with an average runtime of $\widetilde{O}(\log^2 N)$, and furthermore that certifies primes in $\widetilde{O}(\log^2 N)$ with probability $1-1/N^{\alpha/2}$, which is nearly $1$ for large $N$, non-negligible $\alpha$. We first introduce machinery for the trial division step in the next section.

\subsection{Residue Classes in Non-Euclidean Quadratic Rings} 

Let $K=\mathbb{Q}[\sqrt{d}]$ be a quadratic number field, with $d$ squarefree, equipped with the standard norm $|\bullet|$. Let $\mathfrak{I}\subset \mathcal{O}_K$ be an ideal. For $\alpha\in \mathcal{O}_K$, let $\overline{\alpha}$ denote the image of $\alpha$ under the quotient map $\phi_\mathfrak{I}: \mathcal{O}_K\rightarrow \mathcal{O}_K/\mathfrak{I}$. In this section we seek for $\overline{\beta}\in \mathcal{O}_K/\mathfrak{I}$ to find all lifts $\phi_\mathfrak{I}^{-1}(\overline{\beta})$ with norms below a certain bound. 

For the purposes of this paper, we will assume that $\mathfrak{I}=(\iota)$ is principal. We will also restrict to the case of $d<0$ as we will be working with imaginary quadratic number fields. It is well known that for $d<0$, there are only finitely many norm-Euclidean $\mathcal{O}_K$: 

\begin{theorem}
The norm-Euclidean quadratic number fields with $d<0$ are precisely given by 
\begin{align*}
d&=-1,-2,-3,-7,-11. 
\end{align*}
\end{theorem}

For these quadratic number fields alone we can in general guarantee and determine a lift $\phi_\mathfrak{I}^{-1}(\overline{\beta})$ such that $|\phi_\mathfrak{I}^{-1}(\overline{\beta})|<|\iota|$. For other $d$, such a lift may not exist. We consider in this section the following theorem. 

\begin{theorem}\label{resi}
Let $K=\mathbb{Q}[\sqrt{d}]$ for $d<0$ squarefree. Again let $\overline{\beta}\in \mathcal{O}_K/\mathfrak{I}$ with $\mathfrak{I}=(\iota)$ principal. Then there exist $O(4(-d)+4d^2)$ lifts $\phi_\mathfrak{I}^{-1} (\overline{\beta})$ such that $|\phi_\mathfrak{I}^{-1} (\overline{\beta})|\leq |\iota|$. Furthermore they may be found or ruled out in deterministic $O(4(-d)+4d^2)$ steps. 
\end{theorem}

Note that the lifts $\phi_\mathfrak{I}^{-1}(\overline{\beta})$ are given precisely by $\iota X+\beta$ for $X\in \mathcal{O}_K$ and a choice of lift, $\beta$, since $\mathfrak{I}=(\iota)$. In general, we have the following lemma. 

\begin{lemma}\label{bounding}
Let $\iota,\beta,X\in \mathcal{O}_K$ be any elements. Assume $|X|\geq 4(-d)+4d^2$ and $(-d+d^2)|\iota|\geq |\beta|$. Then $|\iota X+\beta|\geq |\iota|$. 
\end{lemma} 

\begin{proof}
If for a positive constant $k$, $|X|\geq k(-d)+kd^2$, then by multiplicativity of field norms, $|\iota X|\geq (k(-d)+kd^2)|\iota|$. Then we can write $\iota X=a+b\sqrt{d}$ with $a^2+b^2(-d)=|\iota X|\Rightarrow a^2\geq |\iota X|/2$ or $b^2(-d) \geq |\iota X|/2$. Writing $\beta = c+e\sqrt{d}$, we have an analogous inequality. Thus we have that 
\begin{align*}
    |\iota X + \beta|&\geq \max((|a|-|c|)^2,(|b|-|e|)^2d)\text{, since $|x+y\sqrt{d}|\geq x^2,y^2d$ for $d<0$,}\\
    &\geq (\sqrt{|\iota X|/2}-\sqrt{(-d+d^2)|\iota|/2})^2 \text{, since $|\beta|\leq (-d+d^2)|\iota|$,}\\
    &\geq (\sqrt{|\iota| (-kd+kd^2)/2} -\sqrt{(-d+d^2)|\iota|/2})^2 \text{, by assumption,}\\
    &= 1/2 (-1 + d) d (-1 + \sqrt{k})^2 |\iota|.
\end{align*} 
We want $1/2 (-1 + d) d (-1 + \sqrt{k})^2\geq 1$, and for all $d$ it suffices to choose $k=4$. 

\end{proof}



To prove theorem \ref{resi}, the general idea is, given a representative $\overline{\beta}$ of an equivalence class, find a new representative of the same equivalence class, $\beta$, with $(-d+d^2)|\iota|\geq |\beta|$, and then to apply theorem \ref{resi}. Let $\iota=a+b\sqrt{d}$. One may naively attempt to consider the image of $\overline{\beta}$ under the surjection with kernel $|\iota|\in \mathfrak{I}$, but this in worst case has $|\beta|=(|\iota|-1+\sqrt{d}(|\iota|-1))^2\sim |\iota|^2$. Attempting to modify Lemma \ref{bounding} to allow this larger bound yields that we must take that $|X|$ grows with $|\iota|$, which prohibits iterating through the lifts $\iota X+\beta$ for large $|iota|$.

Consider instead a coordinate plane with horizontal axis the real line and vertical axis given by real multiples of $\sqrt{d}$. Thus a point $(x,y)$ represents $x+y\sqrt{d}$. Form a new grid with sides $1(\iota)=a+b\sqrt{d},-1(\iota)=-a-b\sqrt{d},\sqrt{d}\iota=bd+a\sqrt{d},$ $-\sqrt{d}\iota=-bd-a\sqrt{d}$ by applying the change of coordinates matrix 
\begin{align*}
    \begin{bmatrix}
           a & -bd\\
           b & -a
    \end{bmatrix}
\end{align*}
to the plane. Each grid square now represents all distinct classes of elements of $\mathcal{O}_K$ modulo $(\iota)$. In particular, in this new coordinate system, moving one step in any direction corresponds to adding a multiple of $\iota$ and thus adding $0 \mod (\iota)$.

To transform $\overline{\beta}=c+e\sqrt{d}$ into these coordinates, solve 

\begin{align*}
    \begin{bmatrix}
           c\\
           e
    \end{bmatrix}&= 
    \begin{bmatrix}
           a\\
           b
    \end{bmatrix} A+
    \begin{bmatrix}
           b(-d)\\
           -a
    \end{bmatrix}B
 \end{align*}
 for $(A,B)$. Then consider $(A',B')=(A-\lfloor{A}\rfloor,B-\lfloor{B}\rfloor)$. Then take 
 
 $$\beta=A'\begin{bmatrix} a\\b\end{bmatrix}+B'\begin{bmatrix}b(-d)\\-a\end{bmatrix}$$
 
 as a representative of the same equivalence class as $\overline{\beta}$ (since we have subtracted elements of $\iota$), that lies within the four grid boxes nearest the origin (those with coordinates $(\pm 1,\pm 1)$). Importantly, both the real and imaginary parts of $\beta$ are bounded above in magnitude by the magnitude of the real and imaginary parts of the four coordinate boxes (given in the new coordinates by $(\pm 1,\pm 1)$) nearest the origin, so: 
\begin{align*}
    |\beta|\leq |\max(|bd|,|a|)+\sqrt{d}\max(|a|,|b|)|&\leq \max(a^2(1-d),b^2d^2(1-d))\\
    &\leq d^2b^2(1-d)-da^2(1-d)\\&=|\iota||-d+d^2|
\end{align*} 
In particular this element $\beta$ satisfies the norm size constraint of Lemma \ref{bounding}. This suffices to prove \ref{resi} as seen in the next section.


\section{Implementation and Runtime}

We now provide the aforementioned Las Vegas algorithm. 

From the discussion in the previous section, we get the following algorithm which proves Theorem \ref{resi}: 

\begin{algorithm}\label{resialg}
    \begin{align*}
        &\text{1. } K=\mathbb{Q}[\sqrt{d}], \overline{\beta}\in \mathcal{O}_K/(\iota), \iota=a+b\sqrt{d}, \overline{\beta}=c+e\sqrt{d}\\
        &\text{2. Compute $(A,B)=$ the solution of $\begin{bmatrix} 
           c\\
           e
         \end{bmatrix}= 
         \begin{bmatrix}
           a\\
           b
         \end{bmatrix} A+
         \begin{bmatrix}
           b(-d)\\
           -a
         \end{bmatrix}B$} \\
         &\text{3. Compute $(A',B')= (A-\lfloor{A}\rfloor,B-\lfloor{B}\rfloor)$ to sufficient precision}\\
         &\text{4. Compute } \beta = A'
         \begin{bmatrix}
           a\\
           b
         \end{bmatrix}+B'
         \begin{bmatrix}
           b(-d)\\
           -a
         \end{bmatrix}\\
         &\text{5. Initialize return values $\{\}$}\\
         &\text{6. For all integers } 0\leq N,M < \sqrt{4(-d)+4d^2},\\ 
         &\text{if $|(N+\sqrt{d}M)\iota+\beta|\leq |\iota|$, append $\{(N,M)\}$ to the result values}\\
         &\text{7. Return the result values}
\end{align*}
\end{algorithm}

\begin{prop}
    Algorithm \ref{resialg} runs in $O(4(-d)+4d^2)$ steps. 
\end{prop}
\begin{proof}
    Inspection. 
\end{proof}

We now seek to implement the Las Vegas test based on Theorem \ref{fast} and the subsequent discussion. We will utilize notation as in Theorem \ref{fast}. Recall that we are assuming an elliptic curve $E/H$ with CM by $\mathcal{O}_K$ and good reduction modulo $\mathfrak{N}$ in precomputation. Before beginning the following algorithm, compute $\psi(\mathfrak{N})^k$ utilizing Definition \ref{comp} of the section Hecke Character Properties. The heuristic runtime of this step is dominated by that of the following algorithm, and it may be done in precomputation for a sequence of integers using reciprocity laws as in \cite[Section 6]{suther} if one wishes to make the step deterministic. 

\begin{algorithm}\label{1}
    Let notation be as above and fix some $\mathfrak{N}\nmid (2)$. Assume $N>1$. Let $(\psi(\mathfrak{N})^k-1)=\Gamma\mathfrak{q}^{e'}$ with $\mathfrak{q}\nmid (2)$ a non-inert principal prime and $q^{e'}:=N_{K/\mathbb{Q}}(\mathfrak{q}^{e'})>N^{1/2+\alpha}$. Further assume $(\psi(\mathfrak{N})^m-1)\not\in \mathfrak{q}^{y}$ for $0< m<k$ and $\mathfrak{q}^y>N^{1/2}$, and that $(q^{y})\nmid ((\psi(\mathfrak{N}))^k-1)$. Choose an $f_k$ as in Theorem \ref{fast}.
    \begin{align*}
        &\text{1. Choose some $Q\in E((\mathcal{O}_H/\mathfrak{N})[x]/(f_k))$}\\
        &\text{2. Compute $P=\left[\frac{(\psi(\mathfrak{N})^k-1)(\overline{\psi(\mathfrak{N})}^k-1)}{{q}^{e_q}}\right]Q\mod \mathfrak{N}$}\\
        &\text{3. Compute and store $[q^x]P \mod \mathfrak{N}$ until $[q^{x+2}]P\mod \mathfrak{N}$ is computed,}\\
        &\text{for $x=0,1,...,e'$ until $[q^x]P\equiv O_E\mod \mathfrak{N}$.}\\
        &\text{If this does not hold for any such $x$, then return \emph{composite}}\\
        &\text{4. Check that $[q^{x-1}]P$ is strongly nonzero modulo $\mathfrak{N}$}\\
        &\text{If this does not hold, then return \emph{composite}}\\
        &\text{5. Check if ${q}^x>N^{1/2}$. If so, return \emph{possibly prime}}\\
        &\text{If not, return \emph{probably composite}}
    \end{align*}
\end{algorithm}

\begin{remark}\label{cute}
We can efficiently compute modulo $\mathfrak{N}$ by simply partially reducing modulo $N=N_{H/\mathbb{Q}}(\mathfrak{N})$. Then when checking whether some point $P=[x_0:y_0:z_0]\equiv O_E\mod \mathfrak{N}$, one can for example check that $a_i\in\mathfrak{N}$ in $\mathcal{O}_H$ for each $a_i$ in $z_0=\sum_{i=0}^{k-1} a_i x^i$. 
\end{remark}

\begin{remark}\label{a}
    A similar algorithm can be developed for the more general case of Theorem \ref{lens}, but complex isogenies must be precomputed for condition $(3)$ of Theorem \ref{lens}, and one must test multiple strongly nonzero conditions, one for each $\mathfrak{q}$. This increasing the complexity of the algorithm to cubic in $\log N$ in the worst case, although for a small number of prime factors $\mathfrak{q}$, say $O(\log \log N)$, the algorithm remains quasi-quadratic. However, Theorem \ref{prob} does not apply and so the algorithm in the general case does not certify primality for almost all choices of input point $Q$. 
\end{remark}

\begin{prop}\label{runtime}
    Algorithm \ref{1} is quasi quadratic in $c=O(\log N)$ for a fixed $k$. 
\end{prop}
\begin{proof}
    For step $2$, use 'binary exponentiation' to compute $P$ as 

    $$P=\bigoplus_j [2^j] Q,$$
    
    where
    
    $$M=\sum_j 2^j = \frac{(\psi(\mathfrak{N})^k-1)(\overline{\psi(\mathfrak{N})}^k-1)}{{q}^{e_q}}$$
    
    To compute a single isogeny $[2]Q$ or $Q_1+Q_2$, use the standard formula for isogenies, and there are at most $2\log M$ such isogenies to compute in this step (compute $[2^x]Q$, $1\leq x\leq j$ and the addition isognies) \cite{silver}. The cost of computing each isogeny is asymptotically the cost of multiplication in $(\mathcal{O}_H/\mathfrak{N})[x]/(f_k)$. The cost of multiplication given remark \ref{cute} and by using the Schonhage-Strassen algorithm on the real and imaginary parts of elements of $\mathcal{O}_H$ modulo $N$ is $O(k^2(2)^2 \log N \log\log N)$ with $K=\sqrt{d}$ \cite{sch,suther}. Since $2\log M$ is $O(k\log N)$ by the Hasse-Weil bound \cite{silver}, this step is $O(\log^2 N \log\log N)$ for fixed $k$. The next isogenies to compute are $[q^x]Q$ which by the same analysis is $O(\log^2 N \log\log N)$. Notice that the other steps are dominated by this. Thus the complexity is $O(c^2 \log c)$ for fixed $k$. 
\end{proof}

If \emph{possibly prime} is returned, run Algorithm \ref{resialg} for $K$, $\psi(\mathfrak{N})^m$, $0\leq m\leq k-1$, and $(\iota)=\mathfrak{q}^{x}$. Check if the norms of any of the results in the result values divide $N$. If not, return \emph{prime}. If one does, return \emph{composite}. These return values are correct by Theorem $2.17$.

\begin{prop}
    In algorithm \ref{1}, if \emph{prime} is returned, then $N$ is prime. If \emph{composite} is returned, then $N$ is composite. 
\end{prop}
\begin{proof}
    The prime case follows directly from Theorem \ref{fast} and the subsequent discussion. If composite is returned in step $3$, then $[(\psi(\mathfrak{N})^k-1)(\overline{\psi(\mathfrak{N})}^k-1)]Q\not\equiv 0\mod \mathfrak{N}$ and so by definition of the Hecke character (Lemma \ref{hecke}), $\mathfrak{N}$ is not prime. If composite is returned in step $4$, then a proper factor of $\mathfrak{N}$ was found. 
\end{proof}

The remaining case is when \emph{probably composite} is returned. In this case by Lemma \ref{prob}, with probability at least $1-1/N^{\alpha/2}$, $N$ is composite. Thus run it through the Miller-Rabin compositeness test, which proves a number is composite in average $O(\log^2 N)$ time \cite{mill,mill2,mill3}. If this does not terminate in $c'\log^2 N$ for some small $c'$, utilize the AKS primality test which runs in $\widetilde{O}(\log^6 N)$ \cite{AKS}.


\begin{theorem}
    Fix $K,k$. Choose a random $Q$ as above. Running Algorithm \ref{1} for an $N$ that satisfies its conditions, along with the subsequent discussion yields a Las Vegas algorithm for primality with average runtime $\widetilde{O}(\log^2 N)$ for large enough $\alpha$. Further if $N$ is prime, it is proven prime in $\widetilde{O}(\log^2 N)$ time for $1-1/N^{\alpha/2}$ of the input parameters $Q$.  
\end{theorem}
\begin{proof}
    The runtime of Algorithm \ref{resialg} is dominated by that of Algorithm \ref{1} for fixed $K=\mathbb{Q}[\sqrt{d}]$. By the above analysis, if \emph{prime} or \emph{composite} is returned before the AKS primality test is used, we have a runtime of $\widetilde{O}(\log^2 N)$. The AKS primality test has to be used for a prime input $N$ with probability less than or equal to $1/N^{\alpha/2}$ by Theorem \ref{prob}. Thus if $N$ is prime the average runtime is $O((1-1/N^{\alpha/2})\widetilde{O}(\log^2 N)+1/N^{\alpha/2}\widetilde{O}(\log^6 N))$. If $\alpha\leq 1$ is large enough, this average runtime is quasi-quadratic. It is known that the average runtime of the Miller Rabin compositeness test is $\widetilde{O}(\log^2 N)$ for composite $N$ \cite{mill,mill2,mill3}. Thus if $N$ is composite, the overall average runtime is quasi-quadratic as well. 

    The latter claim for when $N$ is prime is a direct consequence of Theorem \ref{prob}.
\end{proof}

\subsection*{Example Primality Tests}

We outline one more simplification in this section to make it unnecessary to compute $\mathfrak{N}\subset \mathcal{O}_H$ explicitly.

\begin{remark}\label{rem1}
    Note that class field theory straightforwardly describes the splitting of prime ideals in the Hilbert class field extension $L/K$ of a quadratic imaginary field $K$. In particular, if $\mathfrak{k}$ is prime ideal of $\mathcal{O}_K$, then $\mathfrak{k}$ splits into $h_K/n$ where $h_K$ is the class number of $K$, and $n$ is the order of $\mathfrak{k}$ in the class group $Cl(\mathcal{O}_K)$. 
\end{remark}

In particular if we wish to test the primality of a rational (non-power) integer $M$ that splits into two principal ideals $(\iota),(\overline{\iota})$ in $\mathcal{O}_K$, we cannot be sure that every prime ideal factor of $(\iota)$ is principal. Say $\mathfrak{M}$ lies above $(\iota)$ in $H/K$. If $M$ is prime, however, then $(\iota)$ is a principal prime ideal and by Remark \ref{1}, $(\iota)$ splits completely in $H/K$. Thus $N_{H/K}(\mathfrak{M})=(\iota)$ and $\psi(\mathfrak{M})=\pm \iota$ by Definition \ref{comp}. 

We can utilize Theorem \ref{fast} and Algorithm \ref{1}, assuming $M$ is prime. However since $\mathfrak{N}$ is not computed, we must more subtly check conditions $(1),(2)$. Since $M$ splits completely in $\mathcal{O}_H$, we can carry out computations not on $E/H$ modulo $\mathfrak{M}$ but on the curve $E_b/\mathbb{Q}$ modulo $M$ with coefficients transformed under the isomorphism $\phi: (\mathcal{O}_H/\mathfrak{M})[x]/(f)\rightarrow \mathbb{Z}/M\mathbb{Z}$, which would require precomputing roots modulo $M$ for the sequence of rational integers to test. Alternatively one could do computations on $E/H$ modulo $\mathfrak{M}$ via remark \ref{cute} and checking conditions $(1),(2)$ of Theorem \ref{fast} as follows: 

\begin{lemma}
    Use notation as above and in Theorem \ref{fast}. If $[q^{e_q}]P\equiv O_E\mod M$ then condition $(1)$ is satisfied. Otherwise, write $[q^{e_q}]P\mod M=[x_0:y_0:\sum_{i=0}^{k-1} a_i x^i]$ and check whether $M|N_{H/\mathbb{Q}} (a_i)$ for each $a_i$. If so, condition $(1)$ is satisfied for some $\mathfrak{M}$ above $(\iota)$ in $H/K$; if not, condition $(1)$ does not hold. 
\end{lemma}
\begin{proof}
    The first statement is because $(M)\subset \mathfrak{M}$. Iff $M|N_{H/\mathbb{Q}} (a_i)$ for each $a_i$, then by definition some $[q^{e_q}]P\equiv O_E$ modulo \emph{some} $\mathfrak{M}$ above $(\iota)$. 
\end{proof}

Then condition $(2)$ is modified: 

\begin{lemma}\label{p}
    Condition $(2)$ holds for the same $\mathfrak{M}$ that condition $(1)$ holds for if $[q^{e_q-1}]P$ is strongly nonzero modulo ${M}$.
\end{lemma}
\begin{proof}
    This follows immediately because $\mathfrak{M}|M$. 
\end{proof}

Consider the setup in Theorem \ref{prob}, wherein we must choose $\iota$ so that $(\iota^k-1)=\Gamma\mathfrak{q}^{e'}$ with $q^{e'}>N^{1/2+\alpha}$ and $\mathfrak{q}$ principal. We must also assume that $\iota=\psi(\mathfrak{M})$ for some ideal $\mathfrak{M}$ above $(\iota)$ in $H/K$, so that condition $(1)$ may be satisfied by Lemma \ref{ord}. This has an a priori chance of $1/2$ since $\psi(\mathfrak{M})=\pm \iota$. With these assumptions, utilizing Theorem \ref{prob} and the fact that there are $h_K$ prime ideals $\mathfrak{M}$ above $(\iota)$, there is an $1-h_K/N^{\alpha/2}$ of some $e_q$ satsifying conditions $(1),(2)$ with $N^{1/2}<q^{e_q}\leq q^{e'}$. The rest of Algorithm \ref{1} and the subsequent discussion may be carried out identically. Note that $h_K$ grows roughly as $\sqrt{|-d|}$ for $K=\mathbb{Q}[\sqrt{-d}]$.

This latter method gives us a very flexible framework for testing the primality of certain new sequences of rational integers. All the information required is an elliptic curve $E/H$ with CM by $\mathcal{O}_K$, a rational prime $q=\iota\overline{\iota}$ splitting into two principal ideals over $\mathcal{O}_K$, a primitive $k$-th root of unity $b$ modulo $(\iota)^n$ (which can be computed deterministically given a $k$-th root of unity modulo $(\iota)^{n-1}$ with standard techniques, e.g. in the work of Deng and Lv \cite[Section 4.1]{lv}), some $k$-th power non-residue $a$ modulo $\alpha_n=\gamma\iota^n+b$, and that $E/H$ has good reduction modulo each ideal above $(\alpha_n)$ in $H/K$ (it is sufficient to check $\gcd(N_{K/\mathbb{Q}}(\alpha_n),N_{H/\mathbb{Q}}(\operatorname{disc} (E)))=(1)$). Then we can test the primality of rational integers in the sequence 

$$N_{K/\mathbb{Q}}(\alpha_n)$$ 

when $N_{K/\mathbb{Q}}(\iota^n)>N^{1/2+\alpha}$. In particular, we expect to be able to test and prove the primality of $1/2$ of the rational primes in the sequence (those with $\psi(\mathfrak{M})=\iota$ for $\mathfrak{M}$ above $(\iota)$, so that $(\psi(\mathfrak{M})^k-1)$ is highly factored). 


For an example, consider $K=\mathbb{Q}[\sqrt{-17}]$. By sequence A046085 in the OEIS \cite{oeis}, $K$ has class number $4$. We can write the Hilbert Class field as $H=\mathbb{Q}[\sqrt{-17},\sqrt{(1+\sqrt{17})/2}]$, as is verified in \cite[Example 1.8.14.]{cft}. We can consider the elliptic curve 

$$E:y^2+xy=x^3-36/(j-1728)x-1/(j-1728)$$
with $j=8000(5569095+1350704\sqrt{17}+4\sqrt{3876889241278+940283755330\sqrt{17}}),$ which has CM by $K$. Because $-17\equiv 3\mod 4$, $\mathcal{O}_K=\mathbb{Z}[\sqrt{-17}]$. Notice that $157=(2+3\sqrt{-17})(2-3\sqrt{-17})$ splits into two prime principal ideals in $\mathcal{O}_K$. Consider the sequence 
$$\alpha_n=(2+3\sqrt{-17})^n+14\in \mathcal{O}_K$$
noting that $14^{13}\equiv 1\mod (2+3\sqrt{-17})$ in $\mathcal{O}_K$. We have that $(\alpha_n)$ is prime to $N_{H/K}(\Delta(E))$ by direct computation. First checking that $N_{K/\mathbb{Q}}(\alpha_n)$ is not a power of some rational integer naively, we can test the primality of the sequence $N_{K/\mathbb{Q}}(\alpha_n)$. 


\section{Acknowledgements}

I would like to thank Andrew Sutherland for his guidance on this paper and mentorship in general on elliptic curves, number theory, and algebraic geometry.

\bibliographystyle{plain} 
\bibliography{main}




\end{document}